\documentclass[12pt]{amsart}
\usepackage{fullpage,amssymb}
\usepackage{mathrsfs}
\usepackage{color}
\usepackage{algorithmicx}
\usepackage{algpseudocode}
\usepackage{tikz}

\newtheorem{theorem}{Theorem}[section]
\newtheorem{corollary}[theorem]{Corollary} 
\newtheorem{lemma}[theorem]{Lemma}
\newtheorem{proposition}[theorem]{Proposition}

\theoremstyle{definition}
\newtheorem{definition}[theorem]{Definition}
\newtheorem{remark}[theorem]{Remark}

\newtheorem{problem}[theorem]{Problem}
\newtheorem{algorithm}[theorem]{Algorithm}

\newcommand{\kar} {{\rm char}}

\newcommand{\N}{{\mathbb N}}

\newcommand{\Q}{{\mathbb Q}}
\newcommand{\Z}{{\mathbb Z}}
\newcommand{\Tr}{\operatorname{Tr}}
\newcommand{\trace}{\operatorname{Tr}}
\newcommand{\Mat}{\operatorname{Mat}}
\newcommand{\GL}{\operatorname{GL}}
\newcommand{\SL}{\operatorname{SL}}

\newcommand{\sep}{\operatorname{sep}}
\newcommand{\Set}{{\mathcal S}}

\newcommand{\Ad}{\operatorname{Adj}}

\title{Algorithms for orbit closure separation for invariants and semi-invariants of matrices}
\author{Harm Derksen and Visu Makam}
\thanks{The authors were supported by NSF grants DMS-1601229, DMS-1638352 and CCF-1412958.}

\begin{document}

\maketitle
\begin{abstract}
We consider two group actions on $m$-tuples of $n \times n$ matrices with entries in the field~$K$. The first is simultaneous conjugation by $\GL_n$ and the second is the left-right action of $\SL_n \times \SL_n$. Let $\overline{K}$ be the algebraic closure of the field $K$. Recently, a polynomial time algorithm was found to decide whether $0$ lies in the Zariski closure of the $\SL_n(\overline{K})\times \SL_n(\overline{K})$-orbit of a given $m$-tuple by Garg-Gurvits-Oliveira-Wigderson for the base field $K=\Q$. An algorithm that also works for finite fields of large enough cardinality was given by Ivanyos-Qiao-Subrahmanyam. A more general problem is the {\em orbit closure separation problem} that asks whether the orbit closures of two given $m$-tuples intersect.
For the conjugation action of $\GL_n(\overline{K})$ a polynomial time algorithm for orbit closure separation was given by Forbes and Shpilka in characteristic $0$. 
Here, we give a polynomial time algorithm for the orbit closure separation problem
for both the conjugation action of  $\GL_n(\overline{K})$  and the left-right action of $\SL_n(\overline{K})\times \SL_n(\overline{K})$ in arbitrary characteristic.
  We also improve the known bounds for the degree of separating invariants in these cases.
\end{abstract}

\tableofcontents

\section{Introduction} 
The algorithms we present will only use numbers from the field of definition, as opposed to its algebraic closure (see Section~\ref{not.alg.closed}). However, it will be convenient to assume that the field of definition is algebraically closed for stating and proving results. 

In this paper, let $K$ denote an algebraically closed field. For a vector space $V$ over the field $K$, let $K[V]$ denote the ring of polynomial functions on $V$. 
Suppose that a group $G$ acts on $V$ by linear transformations.
 A polynomial $f \in K[V]$ is called an {\it invariant polynomial}  if it is constant along orbits, i.e., $f(g \cdot v) = f(v)$ for all $g \in G$ and $v\in V$. The invariant polynomials form a graded subalgebra $K[V]^G = \bigoplus_{d=0}^\infty K[V]^G_d,$ where $K[V]^G_d$ denotes the degree $d$ homogeneous invariants. We will call $K[V]^G$ the {\it invariant ring} or the {\it ring of invariants}.

For a point $v \in V$, its orbit $G \cdot v = \{g \cdot v\  |\  g \in G\}$ is not necessarily closed with respect to the Zariski topology. We say that an invariant $f$ separates two points $v,w\in V$ if $f(v)\neq f(w)$. It follows from continuity that any invariant polynomial must take the same value on all points of the closure of an orbit. Hence invariant polynomials cannot separate two points whose orbit closures intersect. 

We can ask the converse question: if $v,w \in V$ such that $\overline{G \cdot v} \cap \overline{G \cdot w} = \emptyset$, then is there an invariant polynomial $f \in K[V]^G$ such that $f(v) \neq f(w)$? The answer to this question is in general negative (see \cite[Example 2.2.8]{DK}).
However, if we enforce additional hypothesis, we get a positive answer as the theorem below shows (see \cite{GIT}).

\begin{theorem} \label{first}
Let $V$ be a rational representation of a reductive group $G$. Then for $v,w \in V$, there exists $f \in K[V]^G$ such that $f(v) \neq f(w)$ if and only if $\overline{G \cdot v} \cap \overline{G \cdot w} = \emptyset$. 
\end{theorem}
Henceforth, we shall assume that $V$ is a rational representation of a reductive group $G$. 
\begin{problem}[orbit closure problem]
Decide whether the orbit closures of two given points $v,w \in V$ intersect.
\end{problem}


\begin{definition}
Two points $v,w \in V$ are said to be {\it closure equivalent} if $\overline{G \cdot v} \cap \overline{G \cdot w} \neq \emptyset$. We write $v \sim w$ if $v$ and $w$ are closure equivalent, and we write $v \not\sim w$ if they are not closure equivalent.
\end{definition} 
By Theorem~\ref{first}, we have $v\sim w$ if and only if $f(v)=f(w)$ for all $f\in K[V]^G$. So $\sim$ is clearly an equivalence relation.
Since closure equivalence can be detected by invariant polynomials, the existence of a small generating set of invariants, each of which can be computed efficiently would give an algorithm for the orbit closure problem. Fortunately, the invariant ring $K[V]^G$ is finitely generated (see \cite{Haboush,Hilbert1,Hilbert2,Nagata}).  

\begin{definition} \label{beta-gen}
We define $\beta(K[V]^G)$ to be the smallest integer $D$ such that invariants of degree $\leq D$ generate $K[V]^G$, i.e., 
$$
\beta(K[V]^G) =\textstyle \min\{D \in \N\ |\ \bigcup_{d=1}^D K[V]^G_{d} \text{ generates } K[V]^G\},
$$
where $\N = \{1,2,\dots\}$.
\end{definition}

We are not just interested in deciding whether orbit closures intersect -- when they do not, we want to provide an explicit invariant that separates them. To be able to do this efficiently, there must exist an invariant of small enough degree that separates the two given points. A strong upper bound on $\beta(K[V]^G)$ would provide evidence that such invariants exist. Such a bound can be obtained for any rational representation $V$ of a linearly reductive group $G$ (see~\cite{Derksen}), but this is often too large. For the cases of interest to us, stronger bounds exist, and we recall them in Section~\ref{preliminaries}. Despite having strong degree bounds, it is a difficult problem to extract a small set of generators. On the other hand, we may only need a subset of the invariants to detect closure equivalence, prompting the definition of a separating set of invariants.

\begin{definition}
A subset of invariants  $\Set \subset K[V]^G$ is called a {\it separating set} of invariants if for every pair $v,w \in V$ such that $v \not\sim w$, there exists $f \in \Set$ such that $f(v) \neq f(w)$.  
\end{definition}

We make another definition.

\begin{definition}
We define $\beta_{\sep} (K[V]^G)$ to be the smallest integer $D$ such that the invariants of degree $\leq D$ form a separating set of invariants, i.e.,
$$
\beta_{\sep} (K[V]^G) =\textstyle  \min \{D \in \N \ |\ \bigcup_{d=1}^D K[V]^G_{d} \text{ is a separating set of invariants}\}.
$$

\end{definition} 

Extracting a small set of separating invariants is also difficult (see \cite{Kemper} for a general algorithm). We now turn to a closely related problem, and to describe this we need to recall the null cone.

\begin{definition}
The null cone $\mathcal{N}(G,V) = \{v \in V\ |\ 0 \in \overline{G \cdot v}\}$.
\end{definition}
For a set of polynomials $I \subset K[V]$ we define its vanishing set 
$$
\mathbb{V}(I) = \{v \in V\  |\  f(v) = 0 \ \text{ for all } f \in I\}.
$$
The null cone can also be defined by $\mathcal{N}(G,V) = \mathbb{V}( K[V]^G_+)$, where  $K[V]^G_+ = \bigoplus_{d=1}^\infty K[V]^G_d$
(see \cite[Definition 2.4.1, Lemma 2.4.2]{DK}).





\begin{problem}[null cone membership problem]
Decide whether a given point $v \in V$ lies in the null cone $\mathcal{N}(G,V)$.
\end{problem}
 Since $0$ is a closed orbit, a point $v \in V$ is in the null cone if and only if $0 \sim v$, and hence the null cone membership problem can be seen as a subproblem of the orbit closure problem. So, the null cone membership problem could potentially be easier than the orbit closure problem. On the other hand, an algorithm for the null cone membership problem may provide a stepping stone for the orbit closure problem. 

In this paper, we are interested in giving efficient algorithms for the orbit closure problem in two specific cases -- matrix invariants and matrix semi-invariants. These two cases have generated considerable interest over the past few years due to their connections to computational complexity, see \cite{GCT, FS13, DM, IQS, IQS2, GGOW, HW}. 

\begin{remark}
For analyzing the run time of our algorithms, we will use the unit cost arithmetic model. This is also often referred to as algebraic complexity. 
\end{remark}

\subsection{Matrix invariants}
Let $\Mat_{p,q}$ be the set of $p\times q$ matrices.  The group $\GL_n$ acts by simultaneous conjugation on the space $V=\Mat_{n,n}^m$ of $m$-tuples of $n\times n$ matrices. This action is given by 
$$
g \cdot (X_1,X_2,\dots,X_m) = (gX_1g^{-1},gX_2g^{-1},\dots,gX_mg^{-1}).
$$

We set $S(n,m) = K[V]^G$. The ring $S(n,m)$ is often referred to as the ring of matrix invariants. We will write $\sim_C$ for
the orbit closure equivalence relation $\sim$ with respect to this simultaneous conjugation action.

\subsubsection{Representation theoretic view point} \label{rep.view}
Orbit closure intersection for matrix invariants has an interpretation in terms of finite dimensional representations of the free algebra. Consider the free algebra $F_m = K \left< t_1,\dots,t_m \right>$ on $m$ indeterminates. An $m$-tuple of matrices $X = (X_1,\dots,X_m)$ gives an $n$-dimensional representation, i.e., an action of $F_m$ on $K^n$ where $t_i$ acts via $X_i$. We will denote this representation by $V_X$. Two $m$-tuples $X$ and $Y$ are in the same $\GL_n$ orbit if and only if $V_X$ and $V_Y$ are isomorphic representations of $F_m$. In other words, we have a correspondence between orbits and isomorphism classes of $n$-dimensional representations of $F_m$. 

Finite dimensional representations of $F_m$ form an abelian category. A representation is called semisimple if it is a direct sum of simple representations. A composition series of a representation $V$ is a filtration $0 = V_0 \subseteq V_1 \subseteq \dots \subseteq V_l = V$ whose successive quotients $V_i/V_{i-1}$ are simple. These simple subquotients are called composition factors and are independent of the choice of composition series. For the representation $V$, the direct sum $\oplus_{i = 1}^l (V_i / V_{i-1})$ is called the associated semisimple representation of $V$. The following statements follow from \cite{Artin}:

\begin{proposition} [\cite{Artin}] \label{Artin}
Consider the simultaneous conjugation action of $G = \GL_n$ on $\Mat_{n,n}^m$, and let $X,Y \in \Mat_{n,n}$. 
\begin{enumerate}
\item The orbit of $X$ is closed if and only if the representation $V_X$ is semisimple. In other words, we have a correspondence between closed orbits and semisimple representations of dimension $n$.
\item There is a unique closed orbit in the orbit closure of $X$, and the representation corresponding to this unique closed orbit is the associated semisimple representation of $V_X$.
\item The orbit closures of $X$ and $Y$ intersect if and only if the associated semisimple representations of $V_X$ and $V_Y$ are isomorphic.
\end{enumerate}
\end{proposition}

For the representation $V_X$, let a composition series be $0 = V_0 \subseteq V_1 \subseteq \dots \subseteq V_l = V_X$. Suppose that  $\dim V_i/V_{i-1} = n_i$ for all $i$. Then for an appropriate choice of basis of $K^n$, all the $X_i$'s are in a block upper triangular form, with the sizes of the diagonal blocks being $n_1,\dots,n_l$. Call $(n_1,\dots,n_l)$ the type of the block upper triangularization. The diagonal blocks correspond to the composition factors $V_i/V_{i-1}$ and the upper triangular blocks capture the information of the non-trivial extensions between these composition factors that make up the module $V_X$. In particular, the associated semisimple representation is then obtained by setting the strictly upper triangular blocks to $0$. Hence, we may also rephrase the orbit closure problem for matrix invariants as follows:

\begin{problem} [Orbit closure for matrix invariants rephrased]
Given $X,Y \in \Mat_{n,n}^m$, decide if there exist $g,h \in \GL_n$ such that the $m$-tuples $g \cdot X$ and $h \cdot Y$ are in block upper triangular form of the same type, such that for all $1 \leq i \leq m$, the diagonal blocks of $(g \cdot X)_i = gX_ig^{-1}$ and $(h \cdot Y)_i = hY_ih^{-1}$ are the same?
\end{problem}

\begin{remark} \label{ass.equal}
The more general question of when two representations $V$ and $W$ of a finitely generated algebra $\mathcal{F}$ have isomorphic associated semisimple representations can be reduced to the above problem. Indeed, we have a surjection $F_m \twoheadrightarrow \mathcal{F}$ for some $m$, and hence $V$ and $W$ can be viewed as representations of $F_m$. $V$ and $W$ have isomorphic associated semisimple representations as $F_m$ representations if and only if they have isomorphic associated semisimple representations as $\mathcal{F}$ representations. 
\end{remark}




\subsubsection{Forbes-Shpilka algorithm} Given any separating set $\Set$, an obvious algorithm for the orbit closure problem would be to evaluate the two given points at every invariant function in the set $\Set$. 
In characteristic $0$, Forbes and Shpilka construct a quasi-polynomial sized set of explicit separating invariants in this case (see \cite{FS13}), but this is not sufficient to get a polynomial time algorithm.

Nevertheless, Forbes and Shpilka give a  deterministic parallel polynomial time algorithm for the orbit closure problem in characteristic $0$.
 Given an input $X \in \Mat_{n,n}^m$, one can construct in polynomial time a noncommutative polynomial $P_X$ with the feature that the coefficients of the monomials in $P_X$ are the evaluations of a generating set of invariants on $X$. Hence, to check if the orbit closures of two points $X,Y \in \Mat_{n,n}^m$ intersect, one needs to determine whether the noncommutative polynomial $P_X - P_Y$ is the zero polynomial. There is an efficient algorithm to test whether $P_X- P_Y$ is the zero polynomial (see \cite{RS}). 

\subsubsection{Our results} 
Forbes and Shpilka's algorithm does not work in positive characteristic. In this paper, we provide an algorithm that works in all characteristics.

\begin{theorem} \label{OCMI}
The orbit closure problem for the simultaneous conjugation action of $\GL_n$ on $\Mat_{n,n}^m$ can be decided in polynomial time. Further, 
if  $A,B \in \Mat_{n,n}^m$ and $A\not\sim_C B$,  then an explicit invariant $f \in S(n,m)$ that separates $A$ and $B$ can be found in polynomial time.
\end{theorem}

Our algorithm has a remarkable and exciting feature -- analyzing it allows us to prove a bound on the degree of separating invariants! The bounds we obtain beat the existing ones in literature, see \cite{GCT}.


\begin{theorem} \label{sep.bound.mi}
We have $\beta_{\sep} (S(n,m)) \leq 4n^2 \log_2(n) + 12n^2 - 4n.$ If we assume $\kar (K) = 0$, then we have $\beta_{\sep}(S(n,m)) \leq 4n \log_2(n) + 12n - 4$.
\end{theorem}

The bound in characteristic $0$ is especially interesting because there are quadratic lower bounds for the degree of generating invariants in this case, see \cite{Formanek}. This also improves the bound in \cite{DM-arbchar} for the degree of invariants defining the null cone.


\subsection{Matrix semi-invariants} \label{msi.intro}
We consider the left-right action of $G=\SL_n \times \SL_n$ on the space $V=\Mat_{n,n}^m$ of $m$-tuples of $n\times n$ matrices. This action is given by
$$
(P,Q)\cdot (X_1,X_2,\dots,X_m)=(PX_1Q^{-1},PX_2Q^{-1},\dots,PX_mQ^{-1}).
$$

We set $R(n,m) = K[V]^G$. The ring $R(n,m)$ is often referred to as the ring of matrix semi-invariants. We will write $\sim_{LR}$ for the equivalence 
relation $\sim$ with respect to this left-right action.

\begin{remark} \label{relate msi-mi}
Two $m$-tuples of $n \times n$ matrices $A = ({\rm Id},A_2,\dots,A_m)$ and $B = ({\rm Id},B_2,\dots,B_m)$ are in the same $\SL_n \times \SL_n$ orbit for the left-right action if and only if $\widetilde{A} = (A_2,\dots,A_m)$ and $\widetilde{B} = (B_2,\dots,B_m)$ are in the same $\GL_n$ orbit for the simultaneous conjugation action. This is compatible with orbit closure in the sense that the orbit closures of $A$ and $B$ intersect for the left-right action if and only if the orbit closures for $\widetilde{A}$ and $\widetilde{B}$ intersect for the simultaneous conjugation action,  see Corollary~\ref{SCtoLR} for the precise statement.
\end{remark}

For $A,B \in \Mat_{n,n}^m$ with $A_1 = {\rm Id}$ it is easy to detect if $A \sim_{LR} B$. If $\det(B_1) \neq 1$, then $A \not\sim_{LR} B$. Otherwise, we have $\det(B_1) = 1$, i.e., $B_1 \in \SL_n$ and hence $\widetilde{B} = (B_1^{-1},{\rm Id}) \cdot B $ is in the same orbit as $B$. Thus, it suffices to detect whether the orbit closures of $A$ and $\widetilde{B}$ intersect. By design, we have $\widetilde{B}_1 = {\rm Id}$. By the above remark, it suffices to detect whether the orbit closures for $(A_2,\dots,A_m)$ and $(\widetilde{B}_2,\dots,\widetilde{B}_m)$ intersect for the conjugation action.

In fact, if we can find a non-singular matrix in the span of $(A_1,\dots,A_m)$, then a similar strategy can be used to detect orbit closure intersection, see Proposition~\ref{LR to SC}. We can now highlight two important issues that need to be addressed.
\begin{enumerate}

\item It is not known how to decide if the span of $A_1,\dots,A_m$ contains a non-singular matrix in polynomial time. In \cite{Val}, it was shown that this problem captures the problem of polynomial identity testing (PIT) (see also \cite{GGOW}). A polynomial time algorithm for PIT is a major open problem in computational complexity.

\item There may not be a non-singular matrix in the span of the matrices $A_1,\dots,A_m$. One might be tempted to hope that this condition would be equivalent to membership in the null cone, but this turns out to be erroneous. The simplest example is the $3$-tuple of $3 \times 3$ matrices 
$$
S = \left( \begin{pmatrix} 0 & 1 & 0 \\ -1 & 0 & 0 \\ 0 & 0 & 0 \end{pmatrix}, \begin{pmatrix} 0 & 0 & 1 \\ 0 & 0 & 0 \\ -1 & 0 & 0 \end{pmatrix}, \begin{pmatrix} 0 & 0 & 0 \\ 0 & 0 & 1 \\ 0 & -1 & 0 \end{pmatrix} \right) \in \Mat_{3,3}^3.
$$
It is well known that $S$ is not in the null cone (see \cite{Dom00b}), but every matrix in the span of $S_1,S_2,S_3$ is singular. Similar examples can be found in \cite{DM-ncrk, DM,  EH, Draisma}. There are several equivalent characterizations of the null cone, and we refer to \cite{GGOW, IQS} for details.
\end{enumerate}



\subsubsection{Null cone membership problem} The null cone membership problem for matrix semi-invariants has attracted a lot of attention due to its connections to non-commutative circuits and identity testing, see \cite{DM,GGOW,HW,IQS,IQS2}. In characteristic $0$, Gurvits' algorithm gives a deterministic polynomial time algorithm, see \cite{DM,GGOW}. There is a different algorithm which works for any sufficiently large field in \cite{IQS2}.

\begin{theorem} [\cite{DM,GGOW,IQS2}]
The null cone membership problem for the left-right action of $\SL_n \times \SL_n$ on $\Mat_{n,n}^m$ can be decided in polynomial time.
\end{theorem}

\subsubsection{Our results} The above theorem allows us to bypass the two issues mentioned above, and we are able to show a polynomial time reduction from the orbit closure problem for matrix semi-invariants to the orbit closure problem for matrix invariants. In fact, the converse also holds, i.e., there is a polynomial time reduction from the orbit closure problem for matrix invariants to the orbit closure problem for matrix semi-invariants.
As a consequence, we have a polynomial time algorithm for the orbit closure problem for matrix semi-invariants as well. Moreover, due to the nature of the reduction, we will be able to find a separating invariant when the orbit closures of two points do not intersect. 
\begin{theorem} \label{OCMSI}
The orbit closure problem for the left-right action of $\SL_n \times \SL_n$ on $\Mat_{n,n}^m$ can be decided in polynomial time. Further for $A,B \in \Mat_{n,n}^m$, if $A \not\sim_{LR} B$, an explicit invariant $f \in R(n,m)$ that separates $A$ and $B$ can be found in polynomial time.
\end{theorem}

In characteristic $0$, an analytic algorithm for the orbit closure problem for matrix semi-invariants has also been obtained by Allen-Zhu, Garg, Li, Oliveira and Wigderson in \cite{ZGLOW}. Our algorithm is algebraic, independent of characteristic, and provides a separating invariant when the orbit closures do not intersect.

In \cite{DM-arbchar}, bounds on $\beta_{\sep} (R(n,m))$ were given. In this paper, we give better bounds using a reduction to matrix invariants.

\begin{theorem} \label{sep.red}
We have $\beta_{\sep} (R(n,m)) \leq n^2  \beta_{\sep}(S(n,mn^2))$.
\end{theorem}

Using the bounds on matrix invariants in Theorem~\ref{sep.bound.mi}, we get bounds for matrix semi-invariants.

\begin{corollary} \label{sep.bound.msi}
We have $\beta_{\sep} (R(n,m)) \leq 4n^4 \log_2(n) + 12n^4 - 4n^3$. If we assume $\kar (K) = 0$, then we have $\beta_{\sep}(R(n,m)) \leq 4n^3 \log_2(n) + 12n^3 - 4n^2$.
\end{corollary}

\begin{remark}
There is a representation theoretic viewpoint for orbit closure intersection for matrix semi-invariants in terms of semistable representations of the $m$-Kronecker quiver. We will not recall it as it is not useful for our purposes and refer the interested reader to \cite{King}.
\end{remark}

\begin{remark}
We will say the null cone membership problem and orbit closure problem for matrix invariants (resp. matrix semi-invariants) to refer to the corresponding problem for the simultaneous conjugation action of $\GL_n$ (resp. left-right action of $\SL_n \times \SL_n$) on $\Mat_{n,n}^m$.
\end{remark}

\begin{remark}
Another interesting problem is to determine if two tuples $(X_1,\dots,X_m)$ and $(Y_1,\dots,Y_m)$ are in the same orbit for the simultaneous conjugation action of $\GL_n$ (also for left-right action). An obvious algorithm to do this would be to solve the equations $X_iZ = ZY_i$ for all $i$. This is a linear system of equations that can be solved efficiently. However, we need such a $Z$ to be invertible, so we would need to be able to verify whether the space of solutions to the equations $X_iZ = ZY_i$ has an invertible matrix in it. As pointed out in the discussion after Remark~\ref{relate msi-mi}, it is not known how to do this in polynomial time. Nevertheless, there is a polynomial time algorithm to test if the two tuples $X$ and $Y$ are in the same orbit! We refer the interested reader to \cite{BL,CIV}. 
\end{remark}

\subsection{Organization} In Section~\ref{preliminaries}, we collect a number of preliminary results on matrix invariants and matrix semi-invariants. In Section~\ref{Section reduction}, we show polynomial time reductions in both directions between the orbit closure problems for matrix invariants and matrix semi-invariants. We give a polynomial time algorithm for finding a basis of a subalgebra of matrices in Section~\ref{Section pivot algo}. In Section~\ref{MI-algo}, we give the algorithm for the orbit closure problem for matrix invariants, and prove bounds on separating invariants. Finally in Section~\ref{Section bounds}, we prove Theorem~\ref{sep.red}.

\section{Preliminaries on matrix invariants and matrix semi-invariants} \label{preliminaries}
\subsection{Matrix invariants}
Let us recall that the ring of matrix invariants $S(n,m)$ is the invariant ring for the simultaneous conjugation action of $\GL_n$ on $\Mat_{n,n}^m$, the space of $m$-tuples of $n \times n$ matrices. Sibirski\u\i\   showed (\cite{Sibirskii}) that in characteristic 0, the ring $S(n,m)$ is generated by traces of words in the matrices, see also \cite{Procesi}.

A word in an alphabet set $\Sigma$ is an expression of the form $i_1i_2 \dots i_k$ with $i_j \in \Sigma$. We denote the set of all words in an alphabet $\Sigma$ by $\Sigma^\star$ (the Kleene closure of $\Sigma$). The set $\Sigma^\star$ includes the empty word $\epsilon$. For a word $w = i_1i_2 \dots i_k$, we define its length $l(w) = k$. For a positive integer $m$, we write $[m] := \{1,2,\dots,m\},$ the set of all positive integers less equal $m$. For a word $w = i_1i_2 \dots i_k \in [m]^\star$, and for $X = (X_1,\dots,X_m) \in \Mat_{n,n}^m$, we define $X_w = X_{i_1}X_{i_2} \dots X_{i_k}$. The function $T_w: \Mat_{n,n}^m \rightarrow K$ given by $T_w(X) := {\rm Tr}(X_w)$ is an invariant polynomial. 

\begin{theorem} [\cite{Sibirskii, Procesi}] \label{Pro}
Assume $\kar (K) = 0$. The invariant functions of the form $T_w$, $w \in [m]^\star$ generate $S(n,m)$.
\end{theorem}

Razmyslov studied trace identities, and as a consequence of his work, we have:

\begin{theorem} [\cite{Raz}] \label{db-mi}
Assume $\kar (K) = 0$. Then $\beta(S(n,m)) \leq n^2$.
\end{theorem}

In positive characteristic, generators of the invariant ring were given by Donkin in \cite{Donkin1,Donkin2}. In simple terms, we have to replace traces with coefficients of characteristic polynomial. For an $n \times n$ matrix $X$, let ${\rm c}(X) = \det({\rm Id} + tX) = \sum_{i=0}^n \sigma_j(X) t^j$ denote its characteristic polynomial. The function $X \mapsto \sigma_j(X)$ is a polynomial in the entries of $X$, and is called the $j^{\rm th}$ characteristic coefficient of $X$. Note that $\sigma_0 = 1$, $\sigma_1(X) =  {\rm Tr}(X)$ and $\sigma_n(X) = \det(X)$. For any word $w$, we define the invariant polynomial $\sigma_{j,w}\in S(n,m)$ by $\sigma_{j,w} (X) := \sigma_j(X_w)$ for $X=(X_1,X_2,\dots,X_m)\in \Mat_{n,n}^m$.

\begin{theorem} [\cite{Donkin1,Donkin2}] \label{donkin-gen}
The set of invariant functions $\{\sigma_{j,w}\mid  w \in [m]^\star,\ 1 \leq j \leq n\}$ is a generating set for the invariant ring $S(n,m)$.
\end{theorem}

In a radically different approach from the case of characteristic $0$, we recently proved a polynomial bound on the degree of generators.

\begin{theorem} [\cite{DM-arbchar}]
We have $\beta(S(n,m)) \leq (m+1)n^4$.
\end{theorem}

\subsection{Matrix semi-invariants} \label{prelim.msi}
The ring of matrix semi-invariants $R(n,m)$ is the ring of invariants for the left-right action of $\SL_n \times \SL_n$ on $\Mat_{n,n}^m$. There is a determinantal description for semi-invariants of quivers, see \cite{DW,DZ,SVd}. Matrix semi-invariants is a special case -- it is the ring of semi-invariants for the generalized Kronecker quiver, for a particular choice of a dimension vector, see for example \cite{DM}.

Given two matrices $A = (a_{ij})$ of size $p \times q$, and $B = (b_{ij})$ of size $r \times s$, we define their tensor (or Kronecker) product to be 
$$A \otimes B = \begin{bmatrix}
a_{11}B   & a_{12}B     &  \cdots     & a_{1n}B  \\
 a_{21}B  &   \ddots               &                 & \vdots  \\
 \vdots      &       &    \ddots  &\vdots   \\
 a_{m1}B  &       \cdots       & \cdots               & a_{mn}B \\
 \end{bmatrix}\in \Mat_{pr,qs}.$$
Associated to each $T=(T_1,T_2,\dots,T_m)\in \Mat_{d,d}^m$, we define a homogeneous invariant $f_{T}\in R(n,m)$ of degree $dn$ by
$$
f_T(X_1,X_2,\dots,X_m)=\det(T_1\otimes X_1+T_2\otimes X_2+\cdots+T_m\otimes X_m).
$$

\begin{theorem} [\cite{DW,DZ,SVd}]
The invariant ring $R(n,m)$ is spanned by all  $f_T$ with $T\in \Mat_{d,d}^m$ and $d\geq 1$.
\end{theorem}

In particular, notice that if $d$ is not a multiple of $n$, then there are no degree $d$ invariants. In other words, we have $R(n,m) = \bigoplus_{d=0}^\infty R(n,m)_{dn}$. A polynomial bound on the degree of generators in characteristic $0$ was shown in \cite{DM}, and the restriction on characteristic was removed in \cite{DM-arbchar}.

\begin{theorem} [\cite{DM,DM-arbchar}]
We have $\beta(R(n,m)) \leq mn^4$. If $\kar (K) = 0$, then $\beta(R(n,m)) \leq n^6$.
\end{theorem}

 Let $\mathcal{N}(n,m)$ denote the null cone for the left-right action of $\SL_n \times \SL_n$ on $\Mat_{n,n}^m$. The following is proved in \cite{DM}.

\begin{theorem} [\cite{DM}] \label{DM}
For $X \in \Mat_{n,n}^m$, the following are equivalent:
\begin{enumerate}
\item $X \notin \mathcal{N}(n,m)$;
\item For some $d \in \N$, there exists $T \in \Mat_{d,d}^m$ such that $f_T(X) \neq 0$;
\item For any $d \geq n-1$, there exists $T \in \Mat_{d,d}^m$ such that $f_T(X) \neq 0.$
\end{enumerate}
\end{theorem}

The above theorem relies crucially on the regularity lemma proved in \cite{IQS}. A more conceptual proof of the regularity lemma is given in \cite{DM-ncrk} using universal division algebras, although it lacks the constructiveness of the original proof. 

An algorithmic version of the above theorem appears in \cite{IQS2}.
\begin{theorem} [\cite{IQS2}] \label{IQS-algo}
For $X \in  \Mat_{n,n}^m$, there is a deterministic polynomial time (in $n$ and $m$) algorithm which determines if $X \notin \mathcal{N}(n,m)$. Further,  for $X \notin \mathcal{N}(n,m)$ and any $n-1 \leq d \leq {\rm poly}(n)$, the algorithm provides in polynomial time, an explicit $T \in \Mat_{d,d}^m$ such that $f_T(X) \neq 0$. 
\end{theorem}

\begin{remark}
We will henceforth refer to  the algorithm in Theorem~\ref{IQS-algo} above as the IQS algorithm.
\end{remark}

For $1 \leq j,k \leq d$, we define $E_{j,k} \in \Mat_{d,d}$ to be the $d \times d$ matrix which has a $1$ in the $(j,k)^{th}$ entry, and $0$ everywhere else.

\begin{definition}
If $X = (X_1,\dots,X_m) \in \Mat_{n,n}^m$, we define $X^{[d]} = (X_i \otimes E_{j,k})_{i,j,k} \in \Mat_{nd,nd}^{md^2}$, where the tuples $(i,j,k) \in [m] \times [d] \times [d]$ are ordered lexicographically.
\end{definition}

\begin{proposition}
The following are equivalent
\begin{enumerate}
\item There exists $f \in R(n,m)$ such that $f(A) \neq f(B)$;
\item There exists $g \in R(nd,md^2)$ such that $g(A^{[d]}) \neq g(B^{[d]})$ for either $d = n-1$ or $d = n$. 
\end{enumerate}
\end{proposition}

\begin{proof}
We first show $(1) \implies (2)$. We can assume $f = f_T$ for some $T \in \Mat_{e,e}^m$ for some $e \geq 1$. Without loss of generality, assume $f(A) \neq 0$. Then we have $\mu = f(B)/f(A) \neq 1$. For any $\mu \neq 1$, both $\mu^{n-1}$ and $\mu^n$ cannot be $1$. Hence for at least one of $d  \in \{n-1,n\}$, we have $\mu^d = f(B)^d/f(A)^d \neq 1$, and hence $f(A)^d \neq f(B)^d$. Now, it suffices to show the existence of $g \in R(nd,md^2)$ such that $g(A^{[d]}) = f(A)^d$ for all $A \in \Mat_{n,n}^m$.

But now, consider \begin{align*}
f_T(A)^d &= \textstyle \det\big(\sum_{i=1}^m T_i \otimes A_i\big)^d \\
&= \textstyle \det\big(\sum_{i=1}^m T_i^{\oplus d} \otimes A_i\big) \\
& = \textstyle\det\big(\sum_{i=1}^m (\sum_{k = 1}^d T_i \otimes E_{k,k} \otimes A_i)\big) \\
& = \textstyle \det\big(\sum_{i,k} T_i \otimes (A_i \otimes E_{k,k})\big).
\end{align*}

Let $S \in \Mat_{e,e}^{md^2}$ given by $S_{i,j,k} = \delta_{j,k} T_i$. We can take $g = f_S$. 

We now show $(2) \implies (1)$. Indeed, we can choose $g = f_S$ for some $S \in \Mat_{e,e}^{md^2}$, $e \geq 1$. We have
\begin{align*}
f_S(A^{[d]}) &=\textstyle \det\big(\sum_{i,j,k}  S_{i,j,k} \otimes (A^{[d]})_{i,j,k}\big)  \\
& =\textstyle \det \big(\sum_{i,j,k} S_{i,j,k} \otimes A_i \otimes E_{j,k}\big) \\
& =\textstyle \det \big( \sum_i \left(\sum_{j,k} S_{i,j,k} \otimes E_{j,k} \right) \otimes A_i\big) \\
& =\textstyle \det \big( \sum_i \widetilde{S}_i \otimes A_i \big),
\end{align*}

where $\widetilde{S}_i = \sum_{j,k} S_{i,j,k} \otimes E_{j,k}$. Let $\widetilde{S} = (\widetilde{S}_1,\dots,\widetilde{S}_m) \in \Mat_{de,de}^m$. Then the above calculation tells us that $f_{\widetilde{S}}(A) = f_S(A^{[d]}) = g(A^{[d]})$. Hence we have $$
f_{\widetilde{S}} (A) = g(A^{[d]}) \neq g(B^{[d]}) = f_{\widetilde{S}} (B).
$$
We can take $f = f_{\widetilde{S}}$.

\end{proof}

\begin{corollary} \label{cor:n-1,n}
The orbit closures of $A$ and $B$ do not intersect if and only if the orbit closures of $A^{[d]}$ and $B^{[d]}$ do not intersect for at least one choice of $d \in \{n-1,n\}$.
\end{corollary}

\subsection{Commuting action of another group} \label{commuting.action}
Let $G$ be a group acting on $V$. Suppose we have another group $H$ acting on $V$, and the actions of $G$ and $H$ commute.
To distinguish the actions, we will denote the action of $H$ by $\star$.  The orbit closure problem for the action of $G$ on $V$ also commutes with the action of $H$. More precisely, we have the following: 

\begin{lemma} \label{commuting}
Let $v,w \in V$ and $h \in H$. Then $v \sim w$ if and only if $h \star v \sim h \star w$.
\end{lemma}

We have a natural identification of $V = \Mat_{n,n}^m$ with $\Mat_{n,n} \otimes K^m$. The latter viewpoint illuminates an action of $\GL_m$ on $V$ that commutes with the left-right action of $\SL_n \times \SL_n$, as well as the simultaneous conjugation action of $\GL_n$. In explicit terms, for $P = (p_{i,j}) \in \GL_m$ and $X = (X_1,\dots,X_m) \in \Mat_{n,n}^m$, we have  
$$
P\star (X_1,\dots,X_m) =\textstyle\big(\sum_ j p_{1,j} X_j, \sum_j p_{2,j} X_j, \dots, \sum_j p_{m,j} X_j\big).
$$

\begin{corollary}
The orbit closure problem for both the left-right action of $\SL_n \times \SL_n$ and the simultaneous conjugation action of $\GL_n$ on $\Mat_{n,n}^m$ commutes with the action of $\GL_m$.  
\end{corollary}

\subsection{A useful surjection} \label{useful.surjection}
We consider the map 
\begin{align*}
\phi :  \Mat_{n,n}^m &   \longrightarrow  \Mat_{n,n}^{m+1} \\
 (X_1,\dots,X_m) & \longmapsto  ({\rm Id},X_1,\dots,X_m)
\end{align*}

This gives a surjection on the coordinate rings $\phi^*: K[\Mat_{n,n}^{m+1}] \rightarrow K[\Mat_{n,n}^m]$, which descends to a surjective map on invariant rings as below (see \cite{Dom00b, DM-arbchar}).  

\begin{proposition} [\cite{Dom00b}] \label{domokos-surjection} 
The map $\phi^* : R(n,m+1) \twoheadrightarrow S(n,m)$ is surjective.
\end{proposition}

We recall the proof of this proposition because the construction in the proof plays a significant role in some of the algorithms below. Before proving the proposition, let us recall some basic linear algebra. For a matrix $X \in \Mat_{n,n}$, let us denote the adjoint (or adjugate) matrix by $\Ad(X)$.

\begin{lemma} \label{adjoint-prop}
Let $X,Y \in \Mat_{n,n}$. Then we have:
\begin{enumerate}
\item $\Ad(XY) = \Ad(Y) \Ad(X)$. 
\item $X \Ad(X) = \det(X)  {\rm Id}$. In particular, if $\det(X) = 1$, then $\Ad(X) = X^{-1}$.
\item For $(P,Q) \in \SL_n \times \SL_n$, we have $\Ad(PXQ^{-1})(PYQ^{-1}) = Q (\Ad(X) Y) Q^{-1}$. 
\end{enumerate}
\end{lemma}

\begin{proof}
The first two are well known. The last one follows from the first two.
\end{proof}

\begin{proof} [Proof of Proposition~\ref{domokos-surjection}]
We want to first show that we have an inclusion $\phi^*(R(n,m+1)) \subseteq S(n,m)$. 

Indeed for $f \in R(n,m+1)$ and $g \in \GL_n$, we have

\begin{align*}
\phi^*(f) (gX_1g^{-1},\dots,gX_mg^{-1}) & = f({\rm Id},gX_1g^{-1},\dots,gX_mg^{-1}) \\
& = f(g{\rm Id}g^{-1},gX_1g^{-1},\dots,gX_mg^{-1}) \\
 & = f ({\rm Id},X_1,\dots,X_m)\\
 & = \phi^*(f) (X_1,\dots,X_m).
\end{align*}

The third equality is the only non-trivial one. Even though $g$ may not be in $\SL_n$, we can replace $g$ by $g' = \lambda g \in \SL_n$ for a suitable $\lambda \in K^*$. Then, one has to observe that conjugation by $g$ and conjugation by $g'$ are the same. 

Now, we show that the image of $\phi^*$ surjects onto $S(n,m)$. For $f \in S(n,m)$, define $\widetilde{f}$ by 
$$
\widetilde{f} (X_1,\dots,X_{m+1}) = f(\Ad(X_1)X_2,\Ad(X_1)X_3,\dots,\Ad(X_1)X_{m+1}).
$$
We claim that $\widetilde{f}$ is invariant w.r.t the left-right action of $\SL_n \times \SL_n$. Indeed for $(P,Q) \in \SL_n \times \SL_n$, we have

\begin{align*}
\widetilde{f} (PX_1Q^{-1},\dots,PX_{m+1}Q^{-1}) &= f(\Ad(PX_1Q^{-1})PX_2Q^{-1},\dots,\Ad(PX_1Q^{-1})PX_{m+1}Q^{-1}) \\
& = f(Q(\Ad(X_1)X_2)Q^{-1},\dots,Q(\Ad(X_1)X_{m+1})Q^{-1}) \\
& = f(\Ad(X_1)X_2,\dots,\Ad(X_1)X_{m+1}) \\
& = \widetilde{f}(X_1,\dots,X_{m+1}).
\end{align*}

The second equality follows from the above lemma, and the third follows because $f$ is invariant under simultaneous conjugation.

Further, we have 
\begin{align*}
(\phi^*(\widetilde{f}))(X_1,\dots,X_m) & = \widetilde{f}({\rm Id},X_1,\dots,X_m)\\
& = f(\Ad({\rm Id})X_1,\dots, \Ad({\rm Id})X_m) \\
& = f(X_1,\dots,X_m)
\end{align*}

Hence for each $f \in S(n,m)$, we have constructed a preimage $\widetilde{f} \in R(n,m+1)$. Thus $\phi^*$ is a surjection from $R(n,m+1)$ onto $S(n,m)$.

\end{proof}

In fact, from the above proof, we can see that for $f \in S(n,m)$, we can construct a pre-image easily. We record this as a corollary.

\begin{corollary} [\cite{Dom00b}]
For $f \in S(n,m)$, the invariant polynomial $\widetilde{f} \in R(n,m+1)$ defined by 
$$
\widetilde{f} (X_1,\dots,X_{m+1}) = f(\Ad(X_1)X_2,\Ad(X_1)X_3,\dots,\Ad(X_1)X_{m+1})
$$
is a pre-image of $f$ under $\phi^*$, i.e., $\phi^*(\widetilde{f}) = f$.

\end{corollary}

\section{Time complexity equivalence of orbit closure problems} \label{Section reduction}
In this section, we will show polynomial reductions between the orbit closure problem for matrix invariants and the orbit closure problem for matrix semi-invariants. We will in fact show a more robust reduction. 

Let $G$ be a group acting on $V$.

\begin{definition}
An algorithm for the {\it orbit closure problem with witness} is an algorithm that decides if $v \sim w$ for any two points $v,w \in V$, and if $v \not\sim w$, provides a witness $f \in K[V]^G$ such that $f(v) \neq f(w)$.
\end{definition}


\subsection{Reduction from matrix invariants to matrix semi-invariants}

Let $A,B \in \Mat_{n,n}^m$. We can consider $\phi(A),\phi(B) \in \Mat_{n,n}^{m+1}$, where $\phi:\Mat_{n,n}^m \rightarrow \Mat_{n,n}^{m+1}$ is the map described in Section~\ref{useful.surjection}. 
\begin{proposition} \label{phi:red}
The following are equivalent:
\begin{enumerate}
\item There exists $f \in S(n,m)$ such that $f(A) \neq f(B)$ 
\item There exists $g \in R(n,m+1)$ such that $g(\phi(A)) \neq g(\phi(B))$.
\end{enumerate}
\end{proposition}

\begin{proof}
Recall the surjection $\phi^*:R(n,m+1) \twoheadrightarrow S(n,m)$ from Proposition~\ref{domokos-surjection}.
Let's first prove $(1) \implies (2)$. Given $f \in S(n,m)$ such that $f(A) \neq f(B)$, take $g$ to be a preimage of $f$, i.e., $\phi^*(g) = f$. Now, 
$$
g(\phi(A)) = \phi^*(g)(A) = f(A) \neq f(B) = \phi^*(g)(B) = g(\phi(B)).
$$

To prove $(2) \implies (1)$, simply take $f = \phi^*(g)$.
\end{proof}

\begin{corollary} \label{SCtoLR}
Let $A,B \in \Mat_{n,n}^m$. Then we have 
$$
A \sim_{C} B \text{ if and only if } \phi(A) \sim_{LR} \phi(B).
$$
\end{corollary}

\begin{corollary}
There is a polynomial reduction that reduces the orbit closure problem with witness for matrix invariants to the orbit closure problem with witness for matrix semi-invariants
\end{corollary}

\begin{proof}
Given $A,B \in \Mat_{n,n}^m$, we construct $\phi(A)$ and $\phi(B)$. Appeal to the orbit closure problem with witness for matrix semi-invariants with input $\phi(A)$ and $\phi(B)$. There are two possible outcomes. If $\phi(A) \sim_{LR} \phi(B)$, then we conclude that $A \sim_{C} B$. If $\phi(A) \not\sim_{LR} \phi(B)$ and $f \in R(n,m+1)$ separates $\phi(A)$ and $\phi(B)$, then $\phi^*(f)$ is an invariant that separates $A$ and $B$. The reduction is clearly  polynomial time.
\end{proof}

\subsection{Reduction from matrix semi-invariants to matrix invariants} \label{cce-msi-mi}

We will show that the orbit closure problem for matrix semi-invariants can be reduced to the orbit closure problem for matrix invariants. Let $A,B \in \Mat_{n,n}^m$. Recall the discussion in Section~\ref{msi.intro}, in particular, that if we can find efficiently a non-singular matrix in the span of $A_1,\dots,A_m$, we would be done. We must address the two issues indicated in Section~\ref{msi.intro}. The IQS algorithm (Theorem~\ref{IQS-algo}) can determine whether $A$ is in the null cone for the left-right action. Further, when $A$ is not in the null cone, it constructs efficiently a non-singular matrix of the form $\sum_{i=1}^m T_i \otimes A_i$, with $T_i \in \Mat_{d,d}$ for any $n-1 \leq d < {\rm poly}(n)$. Roughly speaking, these non-singular matrices will address both issues. We will now make precise statements.

\begin{proposition} \label{LR to SC}
Assume $A,B \in \Mat_{n,n}^m$ such that $\det(A_1) = \det(B_1) \neq 0$. If we denote $\widetilde{A} = (A_1^{-1}A_2,\dots,A_1^{-1}A_m)$ and $\widetilde{B} = (B_1^{-1}B_2,\dots,B_1^{-1}B_m)$, then we have 
$$
A \sim_{LR} B \iff  \widetilde{A} \sim_{C} \widetilde{B}.
$$ 
\end{proposition}

\begin{proof}
Let us first suppose that $\det(A_1) = \det(B_1) = 1$. Then for $ g = (A_1^{-1}, {\rm Id}) \in \SL_n \times \SL_n$, we have $g \cdot A =  ({\rm Id},A_1^{-1}A_2,\dots,A_1^{-1}A_m) = \phi(\widetilde{A})$. Similarly for $h = (B_1^{-1},{\rm Id}) \in \SL_n \times \SL_n$, we have $h \cdot B = \phi(\widetilde{B})$. Now, we have 
$$
A \sim_{LR} B \iff g \cdot A \sim_{LR} h \cdot B \iff \phi(\widetilde{A}) \sim_{LR} \phi(\widetilde{B}) \iff \widetilde{A} \sim_{C} \widetilde{B}.
$$

The last statement follows from Corollary~\ref{SCtoLR}. The general case for $\det(A_1) \neq 0$ follows because the orbit closures of $A$ and $B$ intersect if and only if the orbit closures of $\lambda \cdot A = (\lambda A_1,\dots,\lambda A_m)$ and $\lambda \cdot B = (\lambda B_1, \dots, \lambda B_m)$ intersect for any $\lambda \in K^*$, see Lemma~\ref{commuting}. 
\end{proof}

\begin{lemma} \label{mat.row}
For any non-zero row vector $ {\bf v} = (v_1,\dots,v_m)$, we can construct efficiently a matrix $P \in \GL_m$ such that the top row of the matrix $P$ is ${\bf v}$.
\end{lemma}

\begin{proof}
This is straightforward and left to the reader.
\end{proof}

\begin{algorithm} Now we give an algorithm to reduce the orbit closure problem with witness for matrix semi-invariants to the orbit closure problem with witness for matrix invariants.\\
\noindent{\bf Input:}  $A, B \in \Mat_{n,n}^m$
\begin{description}
\setlength\itemsep{.5em}
\item [Step 1] Check if $A$ or $B$ are in the null cone by the IQS algorithm. If both of them are in the null cone, then $A \sim_{LR} B$. 
If precisely one of them is in the null cone, then $A \not\sim_{LR} B$ and the IQS algorithm gives an invariant that separates $A$ and $B$. If neither are in the null cone, then we proceed to Step 2.

\item [Step 2]  Neither $A$ nor $B$  in the null cone. Now, for $d \in \{n-1,n\}$, the IQS algorithm constructs $T(d) \in \Mat_{d,d}^m$ such that $f_{T(d)}(A) \neq 0$ in polynomial time. We denote $f_d:=f_{T(d)}$. If $f_d(A) \neq f_d(B)$, then $A \not\sim_{LR} B$ and $f_d$ is the separating invariant. Else $f_d(A) = f_d(B)$ for both choices of $d \in \{n-1,n\}$, and we proceed to Step~3.

\item [Step 3] For $d \in \{n-1,n\}$, we have 
\begin{align*}
f_d(A) &= \det \left(\sum_i T(d)_i \otimes A_i\right) \\
& = \det\left(\sum_i (\sum_{j,k}(T(d)_i)_{j,k} E_{j,k}) \otimes A_i\right) \\
& = \det\left(\sum_{i,j,k} (T(d)_i)_{j,k} (E_{j,k} \otimes A_i)\right)\\
&= \det\left(\sum_{i,j,k} (T(d)_i)_{j,k} (A_i \otimes E_{j,k})\right).
\end{align*}

 We can construct efficiently a matrix $P \in \Mat_{md^2,md^2}$ such that the first row is $(T(d)_i)_{j,k})_{i,j,k}$ by Lemma~\ref{mat.row}.  Consider $U = P \star A^{[d]}, V = P \star B^{[d]} \in \Mat_{nd,nd}^{md^2}.$ By construction, this has the property that $\det(U_1) =f_d(A)\neq 0$, and $\det(V_1) = f_d(B)$. Since we did not terminate in Step 2, we know that $\det(U_1) = \det(V_1)$. Let us recall that by Corollary~\ref{cor:n-1,n}, $A \sim_{LR} B$ if and only $A^{[d]} \sim_{LR} B^{[d]}$ for both $d = n-1$ and $d = n$. By Lemma~\ref{commuting}, $A^{[d]} \sim_{LR} B^{[d]}$ if and only if $U \sim_{LR} V$.

To decide whether $U \sim_{LR} V$, we do the following. Let $\widetilde{U} = (U_1^{-1}U_2,\dots,U_1^{-1}U_{md^2})$ and $\widetilde{V} = (V_1^{-1}V_2,\dots,V_1^{-1}V_{md^2})$. By Proposition~\ref{LR to SC}, we have $U \sim_{LR} V$ if and only if $\widetilde{U} \sim_{C} \widetilde{V}$. But this can be seen as an instance of an orbit closure problem with witness for matrix invariants. Also note the fact if we get an invariant separating $\widetilde{U}$ and $\widetilde{V}$, the steps can be traced back to get an invariant separating $A$ and $B$.

\end{description}
\end{algorithm}

\begin{corollary}
There is a polynomial time reduction from the orbit closure problem with witness for matrix semi-invariants to the orbit closure problem with witness for matrix invariants.
\end{corollary}

\section{A polynomial time algorithm for finding a subalgebra basis} \label{Section pivot algo}
Let $\{C_1,\dots,C_m\} \subseteq \Mat_{n,n}$ be a finite subset of $\Mat_{n,n}$. Consider the (unital) subalgebra $\mathcal{C} \subseteq \Mat_{n,n}$ generated by $C_1,\dots,C_m$. In other words, $\mathcal{C}$ is the smallest subspace of $\Mat_{n,n}$ containing 
the identity matrix ${\rm Id}$ and the matrices $C_1,\dots,C_m$ that is closed under  multiplication. For a word $i_1i_2\dots i_b$ we define
$C_w=C_{i_1}C_{i_2}\cdot C_{i_b}$. We also define $C_\epsilon={\rm Id}$ for the empty word $\epsilon$.
We will describe a polynomial time algorithm for finding a basis for $\mathcal{C}$. First observe that $\mathcal{C}$ is spanned by $\{ C_w\mid  w\in [m]^\star\}$. While this is an infinite spanning set, we will extract a basis from this, in polynomial time. We define a total order on $[m]^\star$.

\begin{definition}
For words $w_1 = i_1i_2\dots i_b$ and $w_2 = j_1j_2\dots j_c$, we write $w_1 \prec w_2$ if either
\begin{enumerate}
\item $l(w_1) < l(w_2)$ or
\item $l(w_1) = l(w_2)$ and for the smallest integer $m$ for which $i_m \neq j_m$, we have $i_m < j_m$.
\end{enumerate}
\end{definition}

\begin{remark}
If $w \prec w'$, we will say $w$ is smaller than $w'$.
\end{remark}

We call a word $w$ a pivot if $C_w$ does not lie in the span of all $C_u$, $u\prec w$. Otherwise, we call $w$ a non-pivot.

\begin{lemma}
Let $P = \{w\mid w \text{ is pivot}\}$. Then $\{C_w \mid w \in P\}$ is a basis for $\mathcal{C}$. We will call this the pivot basis.
\end{lemma}

\begin{definition}
For words $w = i_1i_2\dots i_b$ and $w' = j_1j_2\dots j_c$, we define the concatenation 
$$
ww' = i_1i_2\dots i_bj_1j_2\dots j_c.
$$
\end{definition}

\begin{lemma} \label{pos.pivots}
If $w$ is a non-pivot, then $x  w  y$ is a non-pivot for all words $x,y \in [m]^\star$. 
\end{lemma}

\begin{proof}
If $w$ is non-pivot, then $C_w = \sum_k a_k C_{w_k}$ for $w_k \prec w$ and $a_k \in K$. Then we have $C_{x  w  y} = \sum_k a_k C_{x  w_k  y}.$ Hence, $x  w y$ is non-pivot as well.
\end{proof}

\begin{corollary} \label{subword.pivot}
Every subword of a pivot word is a pivot.
\end{corollary}




\begin{lemma} \label{imp.bound.length}
The length of the longest pivot is at most $2n\log_2(n) + 4n - 4$.
\end{lemma}

\begin{proof}
This follows from the main result of \cite{Shitov-length}. For a collection $S \subseteq \Mat_{n,n}$, we define $l(S)$ as the smallest integer $k$ such that all the words of length $\leq k$ in $S$ span the subalgebra of $\Mat_{n,n}$ generated by $S$. In particular, if we take $S = \{C_1,\dots,C_m\}$, this means that any pivot word has length at most $l(S)$. Moreover, $l(S) \leq 2n \log_2(n) + 4n- 4$ is the statement of \cite[Theorem~3]{Shitov-length} (a strong improvement over the previous known bound from \cite{Pappacena}). Thus every pivot word has length at most $2n \log_2(n) + 4n - 4$ as required.
\end{proof}





Now, we describe an efficient algorithm to construct the set of pivots.

\begin{algorithm}[Finding a basis for a subalgebra of $\Mat_{n,n}$]\ \\[5pt]
\noindent{\bf Input:} $n\times n$ matrices $C_1,C_2,\dots,C_m$
\begin{description}
\setlength\itemsep{.5em}

\item[Step 1] Set $ t= 1$ and $P = P_0 = [(\epsilon,{\rm Id})]$.

\item [Step 2]  If $P_{t-1}=[w_1,w_2,\dots,w_s]$, define
 $$P_t=[w_11,\dots,w_1m,w_21,\dots,w_2m,\dots,w_s1,\dots,w_sm]$$

\item [Step 3]
 Proceeding through the list $P_t$, check if an entry $(w,C_w)$ is a pivot. This can be done in polynomial time, as we have to simply check if $C_w$ is a linear combination of smaller pivots. If it is a pivot, add it to $P$. If it is not a pivot, then remove it from $P_t$. Upon completing this step, the list $P_t$ contains all the pivots of length $t$, and the list $P$ contains all pivots of length $\leq t$. 

\item[Step 4] If $P_{t}\neq []$, set $t = t+1$ and go back to Step 2.  Else, return $P$ and terminate.
\end{description}
\end{algorithm}

\begin{corollary} \label{algo-pivots}
There is a polynomial time algorithm to construct the set of pivots. Further, this algorithm also records the word associated to each pivot.
\end{corollary}

\begin{proof}
To show that the above algorithm runs in polynomial time, it suffices to show that the number of words we consider is at most polynomial. Indeed, if there are $k$ pivots of length $d$, then we only consider $km$ words of length $d+1$. Since $k \leq n^2$, the number of words we consider in each degree is at most $n^2m$. We only consider words of length up to $2n \log_2(n) + 4n - 4$. Hence, the number of words considered is polynomial (in $n$ and $m$). 
\end{proof}

\section{Orbit closure problem for matrix invariants} \label{MI-algo}
Let $A,B \in \Mat_{n,n}^m$ with $A = (A_1,\dots,A_m)$ and $B = (B_1,\dots,B_m)$.
Define 
$$
C_i=\begin{pmatrix}A_i & 0\\ 0 & B_i\end{pmatrix}
$$
for all $i$. Let ${\mathcal C}$ be the algebra generated by $C_1,C_2,\dots,C_m$. Let $Z_1,Z_2,\dots,Z_s$ be the pivot basis of ${\mathcal C}$
and write
$$
Z_j=\begin{pmatrix}X_j & 0\\ 0 & Y_j\end{pmatrix}
$$
for all $j$.

\begin{proposition} \label{oc-sub}
Suppose $\kar (K) = 0$. Then we have $A \sim_{C} B$ if and only if $\Tr(X_j)=\Tr(Y_j)$ for all $j$.
\end{proposition}
\begin{proof}
Two orbit closures do not intersect if and only if there is an invariant that separates them. By Theorem~\ref{Pro}, the invariant ring is generated by invariants of the form
$X \mapsto \trace X_w$ for some word $w$ in the alphabet $\{1,2,\dots,m\}$. Note that ${\mathcal C}$ is the span of all 
$$C_w=\begin{pmatrix}
A_w & 0\\
0 & B_w \end{pmatrix},$$ 
where $w$ is a word. Now the proposition follows by linearity of trace.
\end{proof}

We will appeal to a result from \cite{CIW} in order to get a version of the above proposition in arbitrary characteristic (see also \cite{Procesi-2}).



\begin{theorem} \label{sep.pos.mi}
We have $A \sim_{C} B$ if and only if $\det({\rm Id}+tX_j)=\det({\rm Id}+t Y_j)$ as a polynomial in $t$ for all $j$.
\end{theorem}

\begin{proof}
Let $F_m$ denote free algebra generated by $m$ elements $f_1,\dots,f_m$. From Section~\ref{rep.view}, recall that $A$ (resp. $B$) gives rise to a representation $V_A$ (resp. $V_B$) of $F_m$. Recall from Proposition~\ref{Artin} that the orbit closures of $A$ and $B$ intersect if and only if $V_A$ and $V_B$ have the same associated semisimple representation. It is clear that for both $V_A$ and $V_B$, the action of $F_m$ factors through the surjection $F_m \rightarrow \mathcal{C}$ given by $f_i \mapsto C_i$.

Thus it suffices to check whether $V_A$ and $V_B$ have the same associated semisimple representation as $\mathcal{C}$-modules, see Remark~\ref{ass.equal}. The theorem now is just the statement of \cite[Corollary~12]{CIW} for the finite dimensional algebra $\mathcal{C}$.
\end{proof}

\begin{proof} [Proof of Theorem~\ref{OCMI}]
Given $A,B \in \Mat_{n,n}^m$, let $C_i = \begin{pmatrix} A_i & 0 \\ 0 & B_i \end{pmatrix}$. Let $\mathcal{C}$ be the subalgebra generated by $C_1,\dots,C_m$. Construct the pivot basis $Z_1,\dots,Z_s$ of $\mathcal{C}$. For all $j$, let $Z_j = \begin{pmatrix} X_j & 0 \\ 0 & Y_j \end{pmatrix}$. Further for each $j$, we have $Z_j = C_{w_j}$ for some word $w_j \in [m]^\star$, and consequently $X_j = A_{w_j}$ and $Y_j = B_{w_j}$.

If $\kar (K) = 0$, we only need to check if $\Tr(X_j) = \Tr(Y_j)$. If they are equal for all $j$, then we have $A \sim_C B$. Else, we have $\Tr(X_j) \neq \Tr(Y_j)$ for some $j$,  i.e., $T_{w_j}(A) \neq T_{w_j}(B)$ and $A \not\sim_C B$.

 For arbitrary characteristic, we need to check instead if $\det({\rm Id} + tX_j) = \det({\rm Id} + tY_j)$ as a polynomial in $t$ for each $j$. But this can be done efficiently.
When $A \not\sim_{C} B$, the algorithm finds $j$ with $1\leq j\leq n$ and $w \in [m]^\star$ such that  $\sigma_{j,w}(A)\neq \sigma_{j,w}(B)$. This means that $\sigma_{j,w} \in S(n,m)$ is an invariant that separates $A$ and $B$. 

\end{proof}

We will now prove the bounds for separating invariants. For $A,B \in \Mat_{n,n}^m$ with $A \not\sim_{C} B$, we will write $C_i = \begin{pmatrix} A_i & 0 \\ 0 & B_i \end{pmatrix}$ and define $\mathcal{C} \subseteq \Mat_{2n,2n}$ to be the subalgebra generated by $C_1,\dots,C_m$. 

\begin{proof} [Proof of Theorem~\ref{sep.bound.mi}] 
Given $A,B \in \Mat_{n,n}$ with $A \not\sim_{C} B$, let $\{C_1,\dots,C_m\} \subseteq \Mat_{2n,2n}$ be as above, and construct the pivot basis for $\mathcal{C}$. We know that the length of every pivot is at most $2(2n) \log_2(2n) + 4(2n) - 4 = 4n \log_2(n) + 12n - 4.$ by Lemma~\ref{imp.bound.length}.

If $\kar (K) = 0$, then an invariant $T_w$ separates $A$ and $B$ for some pivot $w$.  This means there is an invariant of degree $ \deg(T_w) = l(w) \leq 4n \log_2(n) + 12n - 4$ that separates them. 

If $\kar (K) > 0$, we must have $\det({\rm Id} + t A_w) \neq \det({\rm Id} + tB_w)$ for some pivot $w$. Hence for some $1 \leq j \leq n$, $\sigma_{j,w} (A) \neq \sigma_{j,w} (B)$. This gives an invariant of degree $\leq 4n^2 \log_2(n) + 12n^2 - 4n$ that separates them.
\end{proof}

\begin{remark}
The null cone for the simultaneous conjugation action of $\GL_n$ on $\Mat_{n,n}^m$ is in fact defined by invariants of degree $\leq 2n \log_2(n) + 4n - 4$ in characteristic $0$. To see this, we will use a similar argument as in the proof of Theorem~\ref{sep.bound.mi} above. For $A$ that is not in the null cone, simply consider the subalgebra $\mathcal{A} \subseteq \Mat_{n,n}$ generated by $A_1,\dots,A_m$. For some pivot $w$, the invariant $T_w$ does not vanish on $A$. Every pivot has length at most $2n \log_2(n) + 4n - 4$, so this gives the bound on the null cone. Similarly, in positive characteristic, we can get a bound of $2n^2 \log_2(n) + 4n^2 - 4n$, but better bounds are already known, see \cite{DM-arbchar}.
\end{remark}

\subsection{Non-algebraically closed fields} \label{not.alg.closed}
Suppose $L$ is a subfield of (an algebraically closed field) $K$, and suppose $A,B \in \Mat_{n,n}^m(L)$. Let us assume $L$ is infinite and that we use the unit cost arithmetic model for operations in $L$. 

First, we observe that the entire algorithm for both matrix invariants and matrix semi-invariants can be run using only operations in $L$, and is polynomial time in this unit cost arithmetic model. However, we should point out that the algorithm does not check whether the orbit closures of $A$ and $B$ for the action of $\GL_n(L)$ intersect. Instead, it checks whether the orbit closures of $A$ and $B$ for the action of $\GL_n(K)$ intersect. 

Finally, if we take $L = \Q$, the run times of our algorithms for matrix invariants as well as matrix semi-invariants will be polynomial in the bit length of the inputs.

\begin{remark} We can relax the hypothesis on $L$ by asking for $L$ to be sufficiently large. For fields that are too small, the algorithms will run into issues -- for example, the IQS algorithm (Theorem~\ref{IQS-algo}) requires a sufficiently large field.
\end{remark}

\section{Bounds for separating matrix semi-invariants} \label{Section bounds}
The reduction given in Section~\ref{cce-msi-mi} is  good enough for showing that the orbit closure problems for matrix invariants and matrix semi-invariants are in the same complexity class. In this section we give a stronger reduction with the aim of finding better bounds for the degree of separating invariants for matrix semi-invariants. This reduction can also be made algorithmic, and can replace the reduction in Section~\ref{cce-msi-mi}. However, we will only focus on obtaining bounds for separating invariants.

Let $T \in \Mat_{d,d}^m$. For $X \in \Mat_{n,n}^m$, consider 
$$
L_T(X) = \sum_{k=1}^m T_k \otimes X_k  = \begin{pmatrix} L_{1,1}(X) & \dots & L_{1,d}(X) \\ \vdots & \ddots & \vdots \\ L_{d,1}(X) & \dots & L_{d,d}(X) \end{pmatrix},
$$
where $L_{i,j}(X)$ represents an $n \times n$ block. From the definition of Kronecker product of matrices, one can check that $L_{i,j}(X) = \sum_{k = 1}^m (T_k)_{i,j} X_k$, i.e., a linear combination of the $X_i$. By definition $f_T(X) = \det(\sum_{k=1}^m T_k \otimes X_k) = \det(L_T(X))$. Let 
$$
M_T(X) = \Ad(L_T(X)) = \begin{pmatrix} M_{1,1}(X) & \dots & M_{1,d}(X) \\ \vdots & \ddots & \vdots \\ M_{d,1}(X) & \dots & M_{d,d}(X) \end{pmatrix},
$$
where $M_{i,j}(X)$ represents an $n \times n$ block. The entries of $M_T(X)$ are not linear in the entries of the matrices $X_k$. Instead the entries are  polynomials of degree $dn-1$ in the $(X_k)_{i,j}$'s. We first compute how $M_{i,j}$ change under the action of $\SL_n \times \SL_n$. 

\begin{lemma}
Let $\sigma = (P,Q^{-1}) \in \SL_n \times \SL_n$. Then we have $M_{i.j}(\sigma \cdot X) = Q^{-1}M_{i,j}(X)P^{-1}$.
\end{lemma}

\begin{proof}
First, observe that $L_T(\sigma \cdot X) = (P \otimes {\rm Id}) L_T(X) (Q \otimes {\rm Id})$ follows because $L_T(X)$ is a block matrix where each block is a linear combination of the $X_i$'s. Thus we have 
\begin{align*}
M_T(\sigma \cdot X) & = \Ad(L_T(\sigma \cdot X)) \\
 & = \Ad((P \otimes {\rm Id}) L_T(X) (Q \otimes {\rm Id})) \\
 & = \Ad(Q \otimes {\rm Id} )M_T(X) \Ad(P \otimes {\rm Id} )\\
 & = (Q^{-1} \otimes {\rm Id}) M_T(X) (P^{-1} \otimes {\rm Id})
\end{align*}
The last equality follows from Lemma~\ref{adjoint-prop} because $\det(P \otimes {\rm Id}) = \det(Q \otimes {\rm Id}) = 1$. We deduce that $M_{i.j}(\sigma \cdot X) = Q^{-1}M_{i,j}(X)P^{-1}$.
\end{proof}

For $X \in \Mat_{n,n}^m$, let us define 

$$
X_{i,j,k} = X_k M_{i,j}(X),
$$
for $1\leq k \leq m$ and $1 \leq i,j \leq d$.

The $X_{i,j,k}$'s have been designed in such a way that the left-right action on $X_i$'s turns into a conjugation action on the $X_{i,j,k}$'s. Further, the entries of $X_{i,j,k}$ are degree $dn$ polynomials in the entries of the $X_l$'s.

\begin{corollary}
$(\sigma \cdot X)_{i,j,k} = P X_{i,j,k} P^{-1}$. 
\end{corollary}

\begin{proof}
It follows from the above lemma that
$$(\sigma \cdot X)_{i,j,k} = (\sigma \cdot X)_k M_{i,j} (\sigma \cdot X) = (PX_k Q)(Q^{-1} M_{i,j}(X) P^{-1}) = P X_{i,j,k} P^{-1}.
$$
\end{proof}

Consider the map $\zeta: \Mat_{n,n}^m \rightarrow \Mat_{n,n}^{md^2}$ given by $X \mapsto (X_{i,j,k})_{i,j,k}$. This gives a map on the coordinate rings $\zeta^*:K[\Mat_{n,n}^{md^2}] \rightarrow K[\Mat_{n,n}^m]$. We note that $\zeta$ is a map of degree $dn$ because the entries of $X_{i,j,k}$ are degree $dn$ polynomials in the entries of the $X_l$'s.

The above corollary can be now reformulated as:

\begin{corollary}
Let $\sigma = (P,Q^{-1}) \in \SL_n \times \SL_n$. Then we have $\zeta(\sigma \cdot X) = P \zeta(X) P^{-1}$.
\end{corollary}

\begin{proposition} \label{break-up}
The map $\zeta^*$ descends to a map on invariant rings $\zeta^*:S(n,md^2) \rightarrow R(n,m)$.
\end{proposition}

\begin{proof}
Let $\sigma = (P,Q^{-1}) \in \SL_n \times \SL_n$. For $g \in S(n,md^2)$, by the above corollary, we have $g(\zeta(\sigma \cdot X)) = g(P\zeta(X)P^{-1}) = g(\zeta(X))$. Now observe that $\zeta^*(g) \in R(n,m)$ since $\zeta^*(g) (\sigma \cdot X) = g(\zeta(\sigma \cdot X)) = g(\zeta(X)) = \zeta^*(g)(X)$.
\end{proof}

Observe that this is a very different map from the one in Proposition~\ref{domokos-surjection}. We will still be able to use it to get separating invariants for left-right action from separating invariants for the conjugation action. We make an obvious observation.
 
\begin{corollary}
Suppose we have $g \in S(n,md^2)$ such that $\zeta^*(g)(A) \neq \zeta^*(g)(B)$, then $A \not\sim_{LR}~B$.
\end{corollary}

\begin{remark} \label{certain conditions}
In order for the above corollary to be useful to get separating invariants, we need to be able to guarantee that separating invariants will arise this way. In other words, for $A \not\sim_{LR} B$, we want $g \in S(n,md^2)$ such that $\zeta^*(g)$ separates $A$ and $B$. We will only be able to do it under certain conditions, but that will be sufficient.
\end{remark}

The first issue to notice is that since $\zeta^*$ is a map of degree $dn$, any homogeneous invariant of the form $\zeta^*(g)$ must have degree $dkn$ for some $k \in \Z_{\geq 0}$. For a graded ring $R = \oplus_{t \in \Z} R_t$, let us define its $k^{th}$ veronese subring $\nu_k(R) := \oplus_{t \in \Z} R_{tk}$. 

\begin{lemma}
We have $\zeta^*: S(n,md^2) \rightarrow \nu_{dn}(R(n,m)) \hookrightarrow R(n,m)$.
\end{lemma}

It is certainly possible that for some $d$, no invariant of degree $dkn$ separates $A$ and $B$. A simple example is given by taking any $A$ not in the null cone, and taking $B$ such that $B_i = \mu_{d} A_i$, where $\mu_{d}$ is a $d^{th}$ root of unity for some $d$ coprime to $n$. Hence, we may have to consider more than one choice of $d$. 

For the following lemma, any two coprime numbers can be used in place of $n-1$ and $n$, but this is the smallest pair of coprime numbers larger than $n-1$. The significance of $n-1$ is that as long as $d \geq n-1$, for any $A$ not in the null cone, we can guarantee the existence of an invariant $f_T$, with $T \in \Mat_{d,d}^m$ such that $f_T(A) \neq 0$, see Theorem~\ref{DM}.

\begin{lemma} \label{penultimate}
Assume $A,B \in \Mat_{n,n}^m$ and assume $A \not\sim_{LR} B$. Then $\bigcup\limits_{d \in \{n-1,n\}}\nu_{dn} (R(n,m))$ form a set of separating invariants.
\end{lemma}

\begin{proof}
Since $A \not\sim_{LR} B$, there is a choice of $S \in \Mat_{k,k}^m$, for some $k \geq 1$, such that $f_S(A) \neq f_S(B)$. Without loss of generality, assume $f_S(B) \neq 0$. Hence $f_S(A) / f_S(B) \neq 1$. Once again we must have $f_S(A)^d/ f_S(B)^d \neq 1$ for at least one choice of $d \in \{n-1,n\}$. In particular, for such a $d$, $(f_S)^d \in \nu_{dn}(R(n,m))$ separates $A$ and $B$. 
\end{proof}

Once we have $d$ such $\nu_{dn}(R(n,m))$ separates $A$ and $B$, we still need to produce such an invariant that separates $A$ and $B$. Once, we restrict our attention to invariants whose degree is a multiple of $dn$, the best case scenario is that there is a degree $dn$ invariant that separates $A$ and $B$. We will construct an invariant of the form $\zeta^*(g)$ that separates $A$ and $B$ when degree $dn$ invariants fail to separate $A$ and $B$. The following lemma completes the strategy outlined in Remark~\ref{certain conditions}.

\begin{lemma} \label{last}
Let $A,B \in \Mat_{n,n}^m$ such that $A \not\sim_{LR} B$. Suppose we have $d \geq n-1$ such that $\nu_{dn}(R(n,m))$ separates $A$ and $B$. Then $R(n,m)_{dn} \cup \zeta^*(S(n,md^2))$ will separate $A$ and $B$. 
\end{lemma}

\begin{proof}
Assume that $R(n,m)_{dn}$ fails to separate $A$ and $B$. We will find $g \in S(n,md^2)$ such that $\zeta^*(g)$ separates $A$ and $B$.

Since both $A$ and $B$ cannot be in the null cone, we can assume without loss of generality that $A$ is not in the null cone. By Theorem~\ref{DM}, we have $T \in \Mat_{d,d}^m$, such that $f_T(A) \neq 0$. Now, since degree $dn$ invariants fail to separate $A$ and $B$, we must have $f_T(A) = f_T(B) \neq 0$. 

There exists $U \in \Mat_{dk,dk}^m$ such that $f_U(A) \neq f_U(B)$ since such invariants span $\nu_{dn}(R(n,m))$, which by assumption separates $A$ and $B$.  Now for $X \in \Mat_{n,n}^m$, define
$\mathcal{L}(X) := \sum_{k=1}^m U_k \otimes X_k$ and $\mathcal{R}(X) := {\rm Id_k} \otimes M_T(X)$. Let
$$
N(X) := \mathcal{L}(X) \mathcal{R}(X) = \left(\sum_{k=1}^m U_k \otimes X_k\right) \left( {\rm Id_k} \otimes M_T(X)  \right)
$$

Let us make some observations to help understand $N(X)$.
\begin{itemize}
\item The matrix $\mathcal{L}(X) = \sum_{k=1}^m U_k \otimes X_k$ can be seen as a $dk \times dk$ block matrix, where each block has size $n \times n$. Further, each block is a linear combination of the $X_k$'s. 

\item The matrix $\mathcal{R}(X) = {\rm Id_k} \otimes M_T(X)$ can be seen as a $k \times k$ block matrix, where the off diagonal blocks are $0$, and the diagonal blocks are a copy of $M_T(X)$. Observe further that $M_T(X)$ is a $d \times d$ block matrix, where each block $M_{i,j}$ is of size $n \times n$ as shown above. Hence, we can see $\mathcal{R}(X)$ as a $dk \times dk$ block matrix, where each block is of size $n \times n$ and is either $M_{i,j}$ or $0$.

\item A product of a block from $\mathcal{L}(X)$ and a block from $\mathcal{R}(X)$ yields a linear combination of terms of the form $X_k M_{i,j}$'s, i.e., a linear combination of the $X_{i,j,k}$'s. 

\item We can obtain $N(X)$ as a $dk \times dk$ block matrix by block multiplying $\mathcal{L}(X)$ and $\mathcal{R}(X)$. Hence, we see that each block of $N(X)$ is a linear combination of the $X_{i,j,k}$'s.

\end{itemize}

To summarize, $N(X)$ is a $dk \times dk$ block matrix and the size of each block is $n \times n$. Further, the $(p,q)^{th}$ block $N(X)_{p,q}$ is a linear combination $\sum_{i,j,k} \lambda_{p,q}^{i,j,k} X_{i,j,k}$ for some $\lambda_{p,q}^{i,j,k} \in K$. Now we can define an invariant $g \in S(n,md^2)$. For $Z = (Z_{i,j,k})_{i,j,k} \in \Mat_{n,n}^{md^2}$, we define $N_Z$ to be the $dk \times dk$ block matrix, where the $(p,q)^{th}$ block is given by $\sum_{i,j,k} \lambda_{p,q}^{i,j,k} Z_{i,j,k}$. Let $g(Z) = \det(N_Z)$. This is the required $g$. The point to note here is that by construction, we have $N_{\zeta(X)} = N(X)$. Thus $\zeta^*(g)(X) = g(\zeta(X)) = \det(N_{\zeta(X)}) = \det(N(X))$.  

There are two things we need to check. First that $g$ as defined is indeed invariant under simultaneous conjugation, and then that $\zeta^*(g)(X) = \det(N(X))$ does separate $A$ and $B$. 

The function $g$ is invariant under the simultaneous conjugation action of $\GL_n$ on $\Mat_{n,n}^{md^2}$ because it is given by the determinant of a block matrix whose blocks are linear combinations of matrices from the input $md^2$-tuple. 

Observe that $\det(\mathcal{L}(X)) = f_U(X)$ and $\det(\mathcal{R}(X)) = \det(M_T(X))^k$, hence $\det(N(X)) = f_U(X) \det(M_T(X))^k$. Recall that $f_T(X) = \det(L_T(X))$, and that $M_T(X) = \Ad(L_T(X))$. Now, since $f_T(A) = f_T(B) \neq 0$, we have that $\det(M_T(A)) = \det(M_T(B)) \neq 0$. 
In particular, since $f_U(A) \neq f_U(B)$, we have $\det(N(A)) \neq \det(N(B))$ as required.

Thus $\zeta^*(g)(A) = \det(N(A)) \neq \det(N(B)) = \zeta^*(g)$ showing that $\zeta^*(g)$ indeed separates $A$ and $B$.
\end{proof}

Now, we can finally prove Theorem~\ref{sep.red}.

\begin{proof} [Proof of Theorem~\ref{sep.red}]
Suppose $A,B \in \Mat_{n,n}^m$ with $A \not\sim_{LR} B$. By Lemma~\ref{penultimate}, for at least one choice of $d \in \{n-1,n\}$, we have that $\nu_{dn}(R(n,m))$ separates $A$ and $B$. Fix this $d$. By Lemma~\ref{last}, either $R(n,m)_{dn}$ or $\zeta^*(S(n,md^2))$ separates $A$ and $B$ . In the former case, we have an invariant of degree $dn \leq n^2$ that separates $A$ and $B$. In the latter case, $\zeta^*(S(n,md^2))$ separates $A$ and $B$ which implies that $S(n,md^2)$ separates $\zeta(A)$ and $\zeta(B)$. Hence, we have an invariant $g \in S(n,md^2)$ of degree $\leq \beta_{\sep}(S(n,md^2))$ such that $g(\zeta(A)) \neq g(\zeta(B))$.

Now, since $\zeta$ is a map of degree $dn$, we have $\zeta^*(g) \in R(n,m)$ is a polynomial of degree  $\deg(g)dn \leq n^2 \beta_{\sep} (S(n,md^2)) \leq n^2\beta_{\sep}(S(n,mn^2))$ that separates $A$ and $B$.
\end{proof}

\begin{remark}
It is easy to see from Theorem~\ref{donkin-gen} that the statement of Theorem~\ref{Pro} holds if we assume $\kar (K)>n$ (see also \cite{Zubkov}). Hence, the statements in Theorem~\ref{sep.bound.mi} and Corollary~\ref{sep.bound.msi} that assumed $\kar (K) = 0$ also hold under the assumption that $\kar (K) > n$.
\end{remark}

\subsection*{Acknowledgements} 
We thank the authors of \cite{ZGLOW} for sending us an early version of their paper. We thank G\'abor Ivanyos and Gregor Kemper for providing useful references. Finally, we thank the anonymous referee for several useful suggestions on improving the exposition.

\ \\[20pt]
\noindent{\sl Harm Derksen\\
Department of Mathematics\\
University of Michigan\\
530 Church Street\\
Ann Arbor, MI 48109-1043, USA\\
{\tt hderksen@umich.edu}}

\ \\[20pt]
\noindent{\sl Visu Makam\\
School of Mathematics\\
Institute for Advanced Study\\
1 Einstein Dr\\
Princeton, NJ 08540, USA\\
{\tt visu@umich.edu}}

   \end{document}